\documentclass[11pt,reqno]{amsart}
\usepackage{amsmath,amssymb}

 \makeatletter
 \evensidemargin  \oddsidemargin
 \makeatother

\begin{document}
\thispagestyle{empty}
 \setcounter{page}{1}

\begin{center}
{\large\bf On the Mazur--Ulam theorem in fuzzy n--normed strictly
convex spaces\vskip.25in

{\bf M. Eshaghi Gordji} \\[2mm]

{\footnotesize Department of Mathematics,
Semnan University,\\ P. O. Box 35195-363, Semnan, Iran\\
[-1mm] Tel:{\tt 0098-231-4459905} \\
[-1mm] Fax:{\tt 0098-231-3354082} \\
[-1mm] e-mail: {\tt madjid.eshaghi@gmail.com}}

{\bf S. Abbaszadeh } \\[2mm]

{\footnotesize Department of Mathematics,
Semnan University,\\ P. O. Box 35195-363, Semnan, Iran\\
[-1mm] e-mail: {\tt s.abbaszadeh.math@gmail.com }}

{\bf \bf Th. M. Rassias} \\[2mm]
{\footnotesize Department of Mathematics, National Technical
University of Athens,\\
 Zografou, Campus 15780 Athens, Greece\\
[-1mm]e-mail: {\tt trassias@math.ntua.gr}} }
\end{center}
\vskip 5mm

\noindent{\footnotesize{\bf Abstract.} In this paper, we  generalize
the Mazur--Ulam theorem in the fuzzy real n-normed strictly convex spaces.\\



  \newtheorem{df}{Definition}[section]
  \newtheorem{rk}[df]{Remark}
   \newtheorem{lem}[df]{Lemma}
   \newtheorem{thm}[df]{Theorem}
   \newtheorem{pro}[df]{Proposition}
   \newtheorem{cor}[df]{Corollary}
   \newtheorem{ex}[df]{Example}

 \setcounter{section}{0}
 \numberwithin{equation}{section}

\vskip .2in

\begin{center}
\section{Introduction}
\end{center}

The theory of isometric began in the classical paper \cite{Maz} by
S. Mazur and S. Ulam who proved that every isometry of a real normed
vector space onto another real normed vector space is a linear
mapping up to translation. The property is not true for normed
complex vector space(for instance consider the conjugation on $\Bbb
C$). The hypothesis of surjectivity is essential. Without this
assumption, Baker \cite{Bak} proved that every isometry from a
normed real space into a strictly convex normed real space is linear
up to translation. A number of the mathematicians have had dealt
with
the Mazur--Ulam theorem.\\
 The main theme of this paper is the proof
of the Mazur--Ulam theorem in a fuzzy n-normed strictly convex
space.\\
 In 1984, Katsaras \cite{Kat} defined a fuzzy norm on a linear
space and at the same year Wu and Fang \cite{Co} also introduced a
notion of fuzzy normed space and gave the generalization of the
Kolmogoroff normalized theorem for  fuzzy topological linear space.
In \cite{Bi}, Biswas defined and studied fuzzy inner product spaces
in linear space. In 1994, Cheng and Mordeson introduced a definition
of fuzzy norm on a linear space in such a manner that the
corresponding induced fuzzy metric is of Kramosil and Michalek type
\cite{Kra}. In 2003, Bag and Samanta \cite{Bag} modified the
definition of Cheng and Mordeson \cite{Che} by removing a regular
condition. They also established a decomposition theorem of a fuzzy
norm into a family of crisp norms and investigated some properties
of fuzzy norms (see \cite{Bag}).\\
 In \cite{Ga,Gaa}, G\"{a}hler introduced a new approach for a theory of 2-norm   and n-norm on a linear space. In
\cite{Gun}, Hendra Gunawan and Mashadi gave a simple way to derive
an (n-1)-norm from the n-norm and realized that any n-normed space
is an (n-1)-normed space. Al. Narayanan and S. Vijayabalaji have
introduced the notion of fuzzy n-normed linear space in \cite{Nar}.
Also, S. Vijayabalaji, N. Thillaigovindan and Y. B. Jun, extended
n-normed linear spaces to fuzzy n-normed linear spaces in
\cite{Vij}.  We  mention here the papers and monographs
\cite{BAN,BO,DA,J-R,K-R,L,P-R1,P-R2,R,R-S,R-SH} and \cite{Z}
concerning the isometries on metric spaces.

\section{ Preliminaries }
In this section, we state some essential definitions and results
which will be needed in the sequel.
\begin{df}
Let X be a real linear space. A function $N : X \times \Bbb R
\longrightarrow [0,1]$ (the so--called fuzzy subset) is said to be a
fuzzy norm on X, if for all $x, y \in X$ and all $s, t \in \Bbb
R$:\\
$(N_1)~~N(x,t)=0$ for $t\leq 0;$\\
$(N_2)~~x=0$ if and only if $N(x,t)=1$ for all $t>0;$\\
$(N_3)~~N(tx,s)=N(x,\frac{s}{|t|})$ if $t\neq0$;\\
$(N_4)~~N(x+y,t+ s)\geq min \{N(x,t), N(y,s)\};$\\
$(N_5)~~N(x,.)$ is non--decreasing function on $\Bbb R$ and $\lim
_{t \to \infty} N(x,t)=1;$\\
$(N_6)~~$ For $x\neq 0,$ $N(x,.)$ is (upper semi) continuous on
$\Bbb R.$\\
\end{df}
The pair $(X,N)$ is called a fuzzy normed linear space. One may
regard $N(x,t)$ as the truth value of the statement "the norm of $x$
is less than or equal to the real number $t$".\\
\begin{df}
Let $n\in \Bbb N$(natural numbers) and let $X$ be a real vector
space of dimension $d\geq n$. A real valued function
$\|\bullet,...,\bullet \|$ on $X\times ...\times X$ satisfying the
following four
properties:\\
$(1) \|x_1 ,..., x_n \|=0$, if and only if $x_1 ,..., x_n$ are
linearly dependent;\\
$(2) \|x_1 ,..., x_n \|$ is invariant under any permutation;\\
$(3) \|x_1 ,..., \alpha x_n \|=|\alpha|\|x_1 ,..., x_n \|$, for any
$\alpha \in \Bbb R$;\\
$(4) \|x_1 ,..., x_{n-1} , y+z\|\leq \|x_1 ,..., x_{n-1} , y
\|+\|x_1
,..., x_{n-1} , z\|$;\\
is called an n-norm on $X$ and the pair $(X,\|\bullet,...,\bullet
\|)$, is called an n-normed space.\\
\end{df}
\begin{df}
Let X be a real linear space over a real field $F$. A fuzzy subset
$N$  of $X^n \times \Bbb R$ ($\Bbb R$ is the set of real numbers) is
called the fuzzy n-normed on $X$, if and only if for every $x_1,
..., x_n, x^{'}_n\in X$:\\
$(nN_1)~~$ For all  $t\in \Bbb R$ with $t\leq 0;$, $N(x_1, ..., x_n, t)=0$;\\
$(nN_2)~~$ For all $t\in \Bbb R$ with $t>0$, $N(x_1, ..., x_n,
t)=1$, if and only if $x_, ..., x_n$ are linearly dependent;\\
$(nN_3)~~N(x_1, ..., x_n, t)$ is invariant under any permutation of
$x_1, ..., x_n$;\\
$(nN_4)~~$ For all  $t\in \Bbb R$ with $t>0$, $N(x_1, ..., cx_n,
t)=N(x_1, ..., x_n, \frac{t}{|c|})$, if $c\neq 0$, $c\in
F$(field);\\
$(nN_5)~~$ For all  $s,t\in \Bbb R$, $N(x_1, ...,
x_n+x^{'}_n,s+t)\geq min \{N(x_1, ..., x_n, t),N(x_1, ..., x^{'}_n,
s)\};$\\
$(N_6)~~$ $N(x_1, ..., x_n, t)$, is left continuous and
non--decreasing function of $t\in \Bbb R$ and
\begin{align*}
\lim _{t \to \infty} N(x_1, ..., x_n, t)=1;  \hspace{9.0cm}
\end{align*}
\end{df}
In this case, the pair  $(X,N)$ is called a fuzzy n-normed linear space.\\
\begin{ex}
Let $(X,\|\bullet, ..., \bullet\|)$ be an n-normed space. We define $$N(x_1, ..., x_n, t):=\left\{%
\begin{array}{ll}
   \frac{t}{t+\|x_1, ..., x_n\|}, & when~~~ t\in \Bbb R ~~~with~~~ t> 0~~~, (x_1, ..., x_n) \in
   X\times ...\times X,\\
    0,~~~ & when~~~ t\leq0, \\
\end{array}%
\right.    $$ Then it is easy to show that $(X,N)$ is a fuzzy
n-normed linear space.
\end{ex}
\begin{df}
A fuzzy n-normed space is called {\it strictly convex}, if and only
if for every $x_1, ..., x_n, x^{'}_n\in X$ and $s,t\in \Bbb R$,
$N(x_1, ..., x_n+x^{'}_n,s+t)= min \{N(x_1, ..., x_n, t),N(x_1, ...,
x^{'}_n, s)\}$ and for any $z_1, ..., z_n\in X$, $N(x_1, ..., x_n,
t)= N(z_1, ..., z_n, s)$ implies that $x_1=z_1, ..., x_n=z_n$ and
$s=t$.
\end{df}
\begin{df}
Let $(X,N)$ and $(Y,N)$ be two fuzzy n-normed spaces. We call
$f:(X,N) \to (Y,N)$ a {\it fuzzy n-isometry}, if and only if
\begin{align*}
N(x_1- x_0 , ..., x_n - x_0 , t)=N(f(x_1)-f(x_0), ...,
f(x_n)-f(x_0), t),  \hspace{7.0cm}
\end{align*}
 for all $x_0, x_1, ..., x_n\in X$ and all $t>0.$
\end{df}
\begin{df}
Let $X$ be a real linear space and $x,y,z$ mutually disjoint
elements of $X.$ Then $x,y$ and $z$ are said to be 2-{\it collinear}
if $y-z=t(x-z)$, for some real number $t.$
\end{df}
\section{ Mazur--Ulam problem }
 In this section we  prove
the Mazur--Ulam theorem in the fuzzy real n-normed strictly convex
spaces. From now on, let $(X, N)$ and $(Y, N)$ be two fuzzy n-normed
strictly convex spaces and $f:(X,N) \to (Y,N)$ be a function.
\begin{lem}
For each $x_1, ..., x_n, x^{'}_n\in X$ and $t\in \Bbb R$,\\
$(i)~~$ $N(x_1, ..., x_n- x^{'}_n, t)= N(x_1, ..., x^{'}_n- x_n,
t)$;\\
$(ii)~~$ $N(x_1, ..., x_i, ..., x_j, ..., x_n, t)= N(x_1, ..., x_i +
\alpha x_j, ..., x_j, ..., x_n, t)$, for all $\alpha \in \Bbb R$;\\
\end{lem}
\begin{proof}
\begin{align*}
N(x_1, ..., x_n- x^{'}_n, t)= N(x_1, ..., (-1)( x^{'}_n -x_n), t)&=
N(x_1, ..., x^{'}_n -x_n, \frac{t}{|-1|})\\&= N(x_1, ..., x^{'}_n-
x_n, t). \hspace{4.0cm}
\end{align*}
To prove $(ii)$, assume that $s,t\in \Bbb R$ and $s,t>0$ and $z=
\frac{1}{\alpha} x_i + x_j$. By using $(i)$ and $(nN_2)$ and
$(nN_6)$, we have
\begin{align*}
N&(x_1, ..., x_i, ..., x_j, ..., x_n, t) \leq N(x_1, ..., x_i, ...,
x_j, ..., x_n, t+s)\\&= N(x_1, ..., \alpha (z- x_j), ..., x_j, ...,
x_n, t+s)\\&= N(x_1, ..., z- x_j, ..., x_j, ..., x_n,
\frac{t+s}{|\alpha|})\\&= min \{N(x_1, ..., z, ..., x_j, ..., x_n,
\frac{t}{|\alpha|}), N(x_1, ..., x_j, ..., x_j, ..., x_n,
\frac{s}{|\alpha|})\}\\&= N(x_1, ..., z, ..., x_j, ..., x_n,
\frac{t}{|\alpha|})\\&= N(x_1, ..., \alpha z, ..., x_j, ..., x_n,
t)\\&= N(x_1, ..., x_i + \alpha x_j, ..., x_j, ..., x_n, t)\\&\leq
N(x_1, ..., x_i + \alpha x_j, ..., x_j, ..., x_n, t+s)\\&= min
\{N(x_1, ..., x_i, ..., x_j, ..., x_n, t), N(x_1, ..., \alpha x_j,
..., x_j, ..., x_n, s)\}\\&= N(x_1, ..., x_i, ..., x_j, ..., x_n, t)
\hspace{7.5cm}
\end{align*}
Hence, $N(x_1, ..., x_i, ..., x_j, ..., x_n, t)= N(x_1, ..., x_i +
\alpha x_j, ..., x_j, ..., x_n, t)$, for all $\alpha \in \Bbb R$.
\end{proof}

\begin{lem} Let $x_0 , x_1 \in X$ be arbitrary and $t>0.$ Then $u= \frac{x_0 +
x_1}{2}$ is the unique element of $X$ satisfying
\begin{align*}
N&(x_1 - u, x_1 - x_n, x_2 - x_n, ..., x_{n-1} - x_n, t)\\&= N(x_0 -
x_n, x_0 - u, x_2 - x_n, ..., x_{n-1} - x_n, t)\\&= N(x_0 - x_n, x_1
- x_n, ..., x_{n-1} - x_n, 2t)  \hspace{7.0cm}
\end{align*}
for every $x_2 , ..., x_n\in X$ and $u$, $x_0$ and $x_1$ are
2--colinear.
\end{lem}
\begin{proof}
Since $u= \frac{x_0 + x_1}{2}$, we can write
\begin{align*}
x_0 - u&= x_0 - \frac{x_0 + x_1}{2}= \frac{x_0}{2}- \frac{x_1}{2}=
\frac{x_0 + x_1 - x_1}{2}- \frac{x_1}{2}\\&= -(x_1 - \frac{x_0 +
x_1}{2})= -(x_1 - u).  \hspace{6.5cm}
\end{align*}
Thus we conclude by the Definition $2.7$ that $u$, $x_0$ and $x_1$
are 2--colinear.\\
By using Lemma $3.1$, we can see that
\begin{align*}
N&(x_1 - u, x_1 - x_n, ..., x_{n-1} - x_n, t)\\&= N(x_1 - \frac{x_0
+ x_1}{2}, x_1 - x_n, ..., x_{n-1} - x_n, t)\\&= N(x_1 - x_0, x_1 -
x_n, ..., x_{n-1} - x_n, 2t)\\&= N(x_0 - x_n, x_1 - x_n, ...,
x_{n-1} - x_n, 2t),  \hspace{6.5cm}
\end{align*}
and similarly
\begin{align*}
N&(x_0 - x_n, x_0 - u, x_2 - x_n, ..., x_{n-1} - x_n, t)\\&= N(x_0 -
x_n, x_1 - x_n, ..., x_{n-1} - x_n, 2t) .  \hspace{6.5cm}
\end{align*}
Now, we prove the uniqueness of $u$.\\
Assume that $v\in X$, satisfies the above properties. Since $v$,
$x_0$ and $x_1$ are 2--colinear, there exists a real number $s$ such
that $v:= sx_0 + (1-s)x_1$. In view of Lemma $3.1$ and Definition
$2.5$, we obtain
\begin{align*}
N&(x_0 - x_n, x_1 - x_n, ..., x_{n-1} - x_n, 2t)\\&= N(x_1 - v, x_1
- x_n, x_2 - x_n, ..., x_{n-1} - x_n, t)\\&= N(x_1 - (sx_0 +
(1-s)x_1), x_1 - x_n, ..., x_{n-1} - x_n, t)\\&= N(x_1 - x_0, x_1 -
x_n, ..., x_{n-1} - x_n, \frac{t}{|s|})\\&= N(x_0 - x_n, x_1 - x_n,
..., x_{n-1} - x_n, \frac{t}{|s|}).  \hspace{6.5cm}
\end{align*}
So, $2t= \frac{t}{|s|}$. Since $t>0$, $|s|= \frac{1}{2}.$ Also
\begin{align*}
N&(x_0 - x_n, x_1 - x_n, ..., x_{n-1} - x_n, 2t)\\&= N(x_0 - x_n,
x_0 - v, x_2 - x_n, ..., x_{n-1} - x_n, t)\\&= N(x_0 - x_n, x_0 -
(sx_0 + (1-s)x_1), x_2 - x_n, ..., x_{n-1} - x_n, t)\\&= N(x_0 -
x_n, x_0 - x_1, x_2 - x_n, ..., x_{n-1} - x_n, \frac{t}{|1-s|})\\&=
N(x_0 - x_n, x_1 - x_n, ..., x_{n-1} - x_n, \frac{t}{|1-s|}).
\hspace{6.5cm}
\end{align*}
So $2t= \frac{t}{|1-s|}$. Hence $\frac{1}{2}= |s|= |1-s|$ and so
$s=\frac{1}{2}$. Thus we obtain that $u= v$ and this complete the
proof.
\end{proof}
\begin{lem}
Let $f:(X,N) \to (Y,N)$ is a fuzzy n-isometry;\\
$(i)~~$ For every $x_0 , x_1 , x_2\in X$, if $x_0$, $x_1$ and $x_2$
are 2--colinear, then $f(x_0)$, $f(x_1)$ and $f(x_2)$ are
2--colinear.\\
$(ii)~~$ If $f(0)= 0$, then for every $z_1 , ..., z_n\in X$ and
$t>0$
$$N(z_1 , ..., z_n , t)= N(f(z_1) , ..., f(z_n) , t)$$
\end{lem}
\begin{proof}
Since $x_0$, $x_1$ and $x_2$ are 2--colinear, there exists a real
number $s$ such that $x_1 - x_0= s(x_2 - x_0)$.So, for each $x_3 ,
..., x_{n+1}\in X$ we have
\begin{align*}
N&(f(x_1) - f(x_0), f(x_3) - f(x_0), ..., f(x_{n+1}) - f(x_0),
t)\\&= N(x_1 - x_0, x_3 - x_0, ..., x_{n+1} - x_0, t)\\&= N(x_2 -
x_0, x_3 - x_0, ..., x_{n+1} - x_0, \frac{t}{|s|})\\&= N(f(x_2) -
f(x_0), f(x_3) - f(x_0), ..., f(x_{n+1}) - f(x_0),
\frac{t}{|s|})\\&= N(s(f(x_2) - f(x_0)), f(x_3) - f(x_0), ...,
f(x_{n+1}) - f(x_0), t),\hspace{6.5cm}
\end{align*}
and by definition $2.5$, we conclude that $f(x_1) - f(x_0)= s(f(x_2)
- f(x_0))$.\\
To prove the property $(ii)$, we can write
\begin{align*}
N(z_1 , ..., z_n , t)&= N(z_1 - 0 , ..., z_n - 0 , t)\\&= N(f(z_1) -
f(0), ..., f(z_n) - f(0), t)\\&= N(f(z_1) , ..., f(z_n) , t).
\hspace{7.5cm}
\end{align*}
\end{proof}
\begin{thm}
 Every fuzzy n-isometry $f: (X, N) \to (Y, N)$ is
affine.
\end{thm}
\begin{proof}
$f: (X, N) \to (Y, N)$ is affine, if the function $g: (X, N) \to (Y,
N)$ defined by $g(x)= f(x)- f(0)$, is linear. Its obvious that $g$
is an n-isometry and $g(0)= 0$. Thus, we may assume that $f(0)= 0$.
Hence, it is enough to show that $f$ is linear.\\
Let $x_0, x_1\in X$. By Lemma $3.1$, for every $x_2 , ..., x_n\in X$
we have
\begin{align*}
N&(f(x_0) - f(x_n), f(x_0) - f(\frac{x_0 + x_1}{2}) , f(x_2) -
f(x_n), ..., f(x_{n-1}) - f(x_n), t)\\&= N(f(x_n) - f(x_0),
f(\frac{x_0 + x_1}{2}) - f(x_0), f(x_2) - f(x_0), ..., f(x_{n-1}) -
f(x_0), t)\\&= N(x_n - x_0, \frac{x_0 + x_1}{2}- x_0, x_2 - x_0,
..., x_{n-1} - x_0, t)\\&= N(x_n - x_0, x_1 - x_0, x_2 - x_0, ...,
x_{n-1} - x_0, 2t)\\&= N(f(x_n) - f(x_0), f(x_1) - f(x_0), f(x_2) -
f(x_0), ..., f(x_{n-1}) - f(x_0), 2t)\\&= N(f(x_0) - f(x_n), f(x_1)
- f(x_n), f(x_2) - f(x_n), ..., f(x_{n-1}) - f(x_n), 2t).
\hspace{6.5cm}
\end{align*}
And we can obtain
\begin{align*}
N&(f(x_1) - f(\frac{x_0 + x_1}{2}) , f(x_1) - f(x_n), f(x_2) -
f(x_n)..., f(x_{n-1}) - f(x_n), t)\\&= N(f(\frac{x_0 + x_1}{2}) -
f(x_1) , f(x_n) - f(x_1), f(x_2) - f(x_1), ..., f(x_{n-1}) - f(x_1),
t)\\&= N(\frac{x_0 + x_1}{2} - x_1 , x_n - x_1, x_2 - x_1, ...,
x_{n-1} - x_1, t)\\&= N(x_0 - x_1 , x_n - x_1, x_2 - x_1, ...,
x_{n-1} - x_1, 2t)\\&= N(f(x_0) - f(x_1) , f(x_n) - f(x_1), f(x_2) -
f(x_1), ..., f(x_{n-1}) - f(x_1), 2t)\\&= N(f(x_0) - f(x_n) , f(x_1)
- f(x_n), f(x_2) - f(x_n), ..., f(x_{n-1}) - f(x_n), 2t).
\hspace{6.5cm}
\end{align*}
By $(i)$ of Lemma $(3.3)$, we obtain that $f(\frac{x_0 + x_1}{2})$,
$f(x_0)$ and $f(x_1)$ are 2--colinear. Now, from Lemma $3.2$, we
have
\begin{align*}
f(\frac{x_0 + x_1}{2})= \frac{f(x_0)}{2} + \frac{f(x_1)}{2}
\hspace{7.5cm}
\end{align*}
for all $x_0 , x_1\in X$. It follows that $f$ is $\Bbb
Q$-linear($\Bbb Q$ is the set of rational numbers). We
have to show that $f$ is $\Bbb R$-linear.\\
Let $r\in \Bbb R^{+}$ and $x\in X$. By $(i)$ of Lemma $(3.3)$,
$f(0)$, $f(x)$ and $f(rx)$ are 2--colinear. Since $f(0)= 0$,there
exists $s\in \Bbb R$ such that $f(rx)= sf(x)$. From $(ii)$ of Lemma
$(3.3)$, for every $x_1 , ..., x_{n-1}$ and $t>0$, we have
\begin{align*}
N(x , x_1 , x_2 , ..., x_{n-1} , \frac{t}{r})&= N(rx , x_1 , ...,
x_{n-1} , t)\\&= N(f(rx) , f(x_1) , f(x_2) , ..., f(x_{n-1}) ,
t)\\&= N(sf(x) , f(x_1) , f(x_2) , ..., f(x_{n-1}) , t)\\&= N(f(x) ,
f(x_1) , f(x_2) , ..., f(x_{n-1}) , \frac{t}{|s|})\\&= N(x , x_1 ,
x_2 , ..., x_{n-1} , \frac{t}{|s|}).  \hspace{6.5cm}
\end{align*}
Hence $s= \pm r$. The proof is completed if $s= r$. If $s= -r$, that
is, $f(rx)= -rf(x)$. Then there exists $q_1 , q_2\in \Bbb Q$ such
that $0 < q_1 < r < q_2$. For each $z_1 , ..., z_n\in X$, we have
\begin{align*}
N&(f(x) , f(z_1) - f(q_2 x) , ..., f(z_{n-1}) - f(q_2 x) ,
\frac{t}{q_2 + r})\\&= N(q_2 f(x) - (-rf(x)) , f(z_1) - f(q_2 x) ,
..., f(z_{n-1}) - f(q_2 x) , t)\\&= N(f(rx) - f(q_2 x) , f(z_1) -
f(q_2 x) , ..., f(z_{n-1}) - f(q_2 x) , t)\\&= N(rx - q_2 x , z_1 -
q_2 x , ..., z_{n-1} - q_2 x , t)\\&= N(x , z_1 - q_2 x , ...,
z_{n-1} - q_2 x , \frac{t}{q_2 - r})\\& \geq N(x , z_1 - q_2 x ,
..., z_{n-1} - q_2 x , \frac{t}{q_2 - q_1})\\&= N(q_1 x - q_2 x ,
z_1 - q_2 x , ..., z_{n-1} - q_2 x , t)\\&= N(f(q_1 x) - f(q_2 x) ,
f(z_1) - f(q_2 x) , ..., f(z_{n-1}) - f(q_2 x) , t)\\&= N(f(x) ,
f(z_1) - f(q_2 x) , ..., f(z_{n-1}) - f(q_2 x) , \frac{t}{q_2 -
q_1}).   \hspace{6.5cm}
\end{align*}
By $(nN_6)$, we have $q_2 + r\leq q_2 - q_1$ which is a
contradiction. Hence $s= r$, that is, $f(rx)= rf(x)$ for all
positive real numbers $r$. Therefore $f$ is $\Bbb R$-linear, as
desired.
\end{proof}

{\small


}
\end{document}